\documentclass[11pt]{amsart}

\usepackage{amsmath}
\usepackage{amsthm}
\usepackage{amssymb}
\usepackage{hyperref}
\hypersetup{colorlinks=true, linkcolor=blue}

\newtheoremstyle{theoremstyle}{\baselineskip}{}{\itshape}{}{\bfseries}{.}{5pt}{\thmnumber{#2\ \,}\thmname{#1}\thmnote{ \mdseries #3}}
\theoremstyle{theoremstyle}
\newtheorem{theorem}{Theorem}[section]
\newtheorem{lemma}[theorem]{Lemma}
\newtheorem{corollary}[theorem]{Corollary}

\newtheoremstyle{examplestyle}{\baselineskip}{}{\normalfont}{}{\bfseries}{.}{5pt}{\thmnumber{#2\ \,}\thmname{\mdseries\itshape #1}}
\theoremstyle{examplestyle}

\numberwithin{equation}{theorem}

\newcommand{\C}{\mathbf C}
\newcommand{\N}{\mathbf N}
\newcommand{\Z}{\mathbf Z}
\newcommand{\tr}{\operatorname{tr}}

\usepackage{color}

\title{Partial trace of a full symmetrizer}
\author{Randall R. Holmes}

\address{Randall R. Holmes,
Department of Mathematics and Statistics,
Auburn University,
Auburn AL,
36849,
USA,
\texttt{holmerr@auburn.edu}}

\begin{document}
\setcounter{section}{-1}
\begin{abstract}
A formula for the partial trace of a full symmetrizer is obtained.  The formula is used to provide an inductive proof of the well-known formula for the dimension of a full symmetry class of tensors.
\end{abstract}

\maketitle
\markright{PARTIAL TRACE OF FULL SYMMETRIZER}

\section{Introduction}
\label{sec:Introduction}

Let \(A\) and \(B\) be vector spaces over the field \(\C\) of complex numbers.  The ``partial trace relative to \(B\)'' of a linear operator \(T\) on the tensor product \(A\otimes B\) is a certain linear operator \(\tr_B T\) on the space \(A\) (see Section \ref{sec:PartialTrace}).

In quantum physics, the partial trace is used in the definition of the density operator for a state on a subsystem.  This operator captures all of the information about the state that can be gained from measurements on the subsystem alone \cite[p.~207]{MR2791092}.

The partial trace has uses in pure mathematics as well.  In this paper, we provide a formula for the partial trace relative to \(V\) of a full symmetrizer of the tensor power \(V^{\otimes m}=V^{\otimes(m-1)}\otimes V\) of a vector space \(V\) (see Corollary \ref{cor:PartialTraceOfSymmetrizer}).

As an application, we use the formula to give a proof by induction of the well-known formula for the dimension of a full symmetry class of tensors (see Theorem \ref{thm:DimensionOfSymmetryClass}).

\section{Partial trace}
\label{sec:PartialTrace}

Let \(V\) be a vector space over \(\C\) with basis \(\{v_1,\dots,v_n\}\).  For each \(i\) the dual vector corresponding to \(v_i\) is the linear map \(v_i^*:V\to\C\) given by \(v_i^*(v_j)=\delta_{ij}\) (Kronecker delta).  Denote by \(L(V)\) the space of linear operators on \(V\).  For \(T\in L(V)\) denote by \(\tr T\) the trace of \(T\), so
\(\tr T=\sum_iv_i^*(T(v_i))\).

Let \(A\) and \(B\) be vector spaces over \(\C\) with bases \(\{a_1,\dots,a_m\}\) and \(\{b_1,\dots,b_n\}\), respectively.  Let \(T\in L(A\otimes B)\).  Define \(\tr_B T\in L(A)\) by
\[
\tr_B T (a_k)=\sum_{i=1}^m\sum_{j=1}^n(a_i\otimes b_j)^*(T(a_k\otimes b_j))a_i.
\]
The linear operator \(\tr_B T\) on \(A\) is the \emph{partial trace} of \(T\) relative to \(B\) \cite[pp.~210--211]{MR2791092}.  If \(A=\C\), the space \(A\otimes B\) identifies with \(B\) and \(\tr_B T=\tr T\).  At the other extreme, if \(B=\C\), the space \(A\otimes B\) identifies with \(A\) and \(\tr_B T=T\).

\begin{lemma}
\label{lem:TraceIsTraceOfPartialTrace}
  For every \(T\in L(A\otimes B)\), we have \(\tr T=\tr(\tr_B T)\).
\end{lemma}
\begin{proof}
Let \(T\in L(A\otimes B)\) and let the notation be as above.  We have
\begin{align*}
\tr(\tr_B T)&=\sum_{k=1}^ma_k^*(\tr_B T(a_k))\\
&=\sum_{k=1}^m\sum_{j=1}^n(a_k\otimes b_j)^*(T(a_k\otimes b_j))\\
&=\tr T.
\end{align*}
\end{proof}

\begin{lemma}
\label{lem:PartialTraceOfProduct}
  For every \(S\in L(A)\) and \(R\in L(B)\), we have \(\tr_B(S\otimes R)=(\tr R)S\).
\end{lemma}
\begin{proof}
Let \(S\in L(A)\) and \(R\in L(B)\) and let the notation be as above.  For each \(1\le k\le m\), we have
\begin{align*}
\tr_B(S\otimes R)(a_k)&=\sum_{i=1}^m\sum_{j=1}^n(a_i\otimes b_j)^*(S(a_k)\otimes R(b_j))a_i\\
&=\sum_{i=1}^m\sum_{j=1}^na_i^*(S(a_k))b_j^*(R(b_j))a_i\\
&=\sum_{j=1}^nb_j^*(R(b_j))\sum_{i=1}^ma_i^*(S(a_k))a_i\\
&=(\tr R)S(a_k),
\end{align*}
and the claim follows.
\end{proof}

\section{Symmetrizer}
\label{sec:Symmetrizer}

For \(l\in\N:=\{0,1,2,\dots\}\), put
\[
\Gamma^+_l=\{\gamma=(\gamma_1,\gamma_2,\dots,\gamma_l,0,0,\dots)\mid \gamma_i\in\N\text{ for each }i\}
\]
and put \(\Gamma^+=\cup_l\Gamma^+_l\).

Let \(\gamma\in\Gamma^+\) and let \(m\) be a positive integer.  We write \(\gamma\models m\) and say that \(\gamma\) is an \emph{improper partition} of \(m\) if \(\sum_i\gamma_i=m\).  We write \(\gamma\vdash m\) and say that \(\gamma\) is a \emph{partition} of \(m\) if \(\gamma\models m\) and \(\gamma_i\ge\gamma_{i+1}\) for each \(i\).

Let \(m\) be a nonnegative integer.  Denote by \(S_m\) the symmetric group of degree \(m\).
The \emph{cycle partition} of a permutation \(\sigma\in S_m\) is the partition of \(m\) obtained by writing the lengths of the cycles in a disjoint cycle decomposition of \(\sigma\) in nonincreasing order (followed by zeros).  Two elements of \(S_m\) are conjugate if and only if they have the same cycle partition (see \cite[p.~9]{MR644144} or \cite[Theorem 3.18]{MR1475219}).

The number of (ordinary) irreducible characters of \(S_m\) is the same as the number of conjugacy classes of \(S_m\), so these irreducible characters are indexed by the partitions of \(m\).  For \(\alpha\vdash m\), denote by \(\chi_\alpha\) the irreducible character of \(S_m\) corresponding to \(\alpha\) as in \cite[p.~36]{MR644144} and \cite[p.~99]{MR1475219}.

Let \(n\) be a positive integer. Put
\[
\Gamma_{m,n}=\{\gamma=(\gamma_1,\gamma_2,\dots,\gamma_m)\in\Z^m\mid1\le\gamma_i\le n\text{ for each }i\}
\]
(which equals the set consisting of the empty tuple \((\,)\) if \(m=0\)). There is a right action of the group \(S_m\) on the set \(\Gamma_{m,n}\) given by place permutation: \(\gamma\sigma:=(\gamma_{\sigma(1)},\dots,\gamma_{\sigma(m)})\) (\(\gamma\in\Gamma_{m,n}\), \(\sigma\in S_m\)).

Let \(V\) be a vector space over \(\C\) of dimension \(n\) and let \(\{e_1,e_2,\dots,e_n\}\) be a basis of \(V\).
The \(m\)th tensor power \(V^{\otimes m}\) of \(V\) has basis \(\{e_\gamma\mid\gamma\in\Gamma_{m,n}\}\), where \(e_\gamma:=e_{\gamma_1}\otimes\cdots\otimes e_{\gamma_m}\) (with \(V^{\otimes0}=\C\) and \(e_{(\,)}=1\)).  For \(\sigma\in S_m\) denote by \(P_m(\sigma)\) the linear operator on \(V^{\otimes m}\) given by \(P_m(\sigma)(e_\gamma)=e_{\gamma\sigma^{-1}}\) (\(\gamma\in\Gamma_{m,n}\)).

Let \(\alpha\vdash m\).  The linear operator \(T_\alpha\) on \(V^{\otimes m}\) given by
\[
T_\alpha=\frac{\chi_\alpha(1)}{m!}\sum_{\sigma\in S_m}\chi_\alpha(\sigma)P_m(\sigma)
\]
is the \emph{symmetrizer} of \(V^{\otimes m}\) corresponding to \(\alpha\) (or to \(\chi_\alpha\)).  This operator is often referred to as a \emph{full symmetrizer} to distinguish it from a symmetrizer corresponding to an irreducible character of a proper subgroup of \(S_m\), which is defined similarly.

From now on, we assume that \(m>0\).  Put \(A=V^{\otimes(m-1)}\) and \(B=V\).  We have \(V^{\otimes m}=A\otimes B\) so the partial trace \(\tr_BT_\alpha\) is defined.  The main result of the paper (see Corollary \ref{cor:PartialTraceOfSymmetrizer}) expresses this partial trace as a linear combination of symmetrizers \(T_\beta:V^{\otimes(m-1)}\to V^{\otimes(m-1)}\) for various \(\beta\vdash m-1\).  The proof of this result requires several preliminaries, to which we now turn.

For \(1\le i\le m\), put
\[
\kappa_i=
\begin{cases}
  n,&i=m,\\
  1,&i\not=m.
\end{cases}
\]

\begin{lemma}
\label{lem:PartialTraceOfGroupElement}
  Let \(\tau\in S_{m-1}\) and \(1\le i\le m\), and put \(\sigma=\tau(i,m)\in S_m\).  We have
  \[
  \tr_B P_m(\sigma)=\kappa_iP_{m-1}(\tau).
  \]
\end{lemma}
\begin{proof}
First assume that \(i=m\), so that \(\sigma=\tau\).  We have \(P_m(\sigma)=P_{m-1}(\tau)\otimes1_V\), so Lemma \ref{lem:PartialTraceOfProduct} gives
\[
\tr_B P_m(\sigma)=(\tr1_V)P_{m-1}(\tau)=nP_{m-1}(\tau)=\kappa_iP_{m-1}(\tau).
\]
Now assume that \(i\not=m\).  For \(\mu\in\Gamma_{m-1,n}\) and \(1\le j\le n\), put
\[
(\mu,j)=(\mu_1,\dots,\mu_{m-1},j)\in\Gamma_{m,n}.
\]
Fix \(\mu\in\Gamma_{m-1,n}\).  From the definition of partial trace, we get
\[
\tr_B P_m(\sigma)(e_\mu)=\sum_{\gamma\in\Gamma_{m-1,n}}\sum_{j=1}^n(e_\gamma\otimes e_j)^*(P_m(\sigma)(e_{(\mu,j)}))e_\gamma.
\]
For \(1\le j\le n\), define \(\mu[j]\in\Gamma_{m-1,n}\) by
\[
\mu[j]_k=
\begin{cases}
  j,&k=i,\\
  \mu_k,&k\not=i.
\end{cases}
\]
and note that \(\mu[\mu_i]=\mu\).  For each \(1\le j\le n\), we have
\[
(\mu,j)\sigma^{-1}=(\mu,j)(i,m)\tau^{-1}=(\mu[j],\mu_i)\tau^{-1}=(\mu[j]\tau^{-1},\mu_i),
\]
so
\[
P_m(\sigma)(e_{(\mu,j)})=e_{(\mu,j)\sigma^{-1}}=e_{(\mu[j]\tau^{-1},\mu_i)}=e_{\mu[j]\tau^{-1}}\otimes e_{\mu_i}.
\]
Returning to the formula above, we now have, using the Kronecker delta a few times,
\begin{align*}
  \tr_B P_m(\sigma)(e_\mu)&=\sum_{\gamma\in\Gamma_{m-1,n}}\sum_{j=1}^n(e_\gamma\otimes e_j)^*(e_{\mu[j]\tau^{-1}}\otimes e_{\mu_i})e_\gamma\\
  &=\sum_{\gamma,j}\delta_{\gamma,\mu[j]\tau^{-1}}\delta_{j,\mu_i}e_\gamma=\sum_\gamma\delta_{\gamma,\mu[\mu_i]\tau^{-1}}e_\gamma\\
  &=e_{\mu\tau^{-1}}=P_{m-1}(\tau)(e_\mu).
\end{align*}
Therefore, \(\tr_B P_m(\sigma)=P_{m-1}(\tau)=\kappa_iP_{m-1}(\tau)\), as claimed.
\end{proof}

\begin{lemma}
\label{lem:CosetReps}
  The set \(S_m\) is the disjoint union of the cosets \(S_{m-1}(i,m)\), \(1\le i\le m\).
\end{lemma}
\begin{proof}
  Let \(\sigma\in S_m\) and put \(i=\sigma^{-1}(m)\).  We have \(\tau:=\sigma(i,m)\in S_{m-1}\) and \(\sigma=\tau(i,m)\in S_{m-1}(i,m)\), so \(S_m\) is the union of the indicated cosets.

  Let \(1\le i,j\le m\) and assume that the cosets \(S_{m-1}(i,m)\) and \(S_{m-1}(j,m)\) intersect.  Then the cosets are equal, implying \((i,m)(j,m)=(i,m)(j,m)^{-1}\in S_{m-1}\), that is, \((i,m)(j,m)\) fixes \(m\).  Therefore, \(i=j\) and we conclude that the indicated cosets are disjoint..
\end{proof}

We will require the next two theorems about the irreducible characters of \(S_m\).  The statements require additional notation.

For \(\beta\in\Gamma^+\), denote by \(l(\beta)\) (length of \(\beta\)) the least \(l\in\N\) for which \(\beta_l\in\Gamma^+_l\).  If \(\beta\) is nonzero, \(l(\beta)\) is the least index \(l\) for which \(\beta_l\not=0\); otherwise, \(l(\beta)=0\).

Any \(\beta\models m\) corresponds, by an arrangement of its entries in nonincreasing order, to a uniquely determined partition of \(m\) and hence to a uniquely determined conjugacy class of \(S_m\), which we refer to as the conjugacy class corresponding to \(\beta\).  For \(\alpha\vdash m\) and \(\beta\models m\) we write \(\chi_\alpha(\beta)\) for the value of the character \(\chi_\alpha\) on the conjugacy class corresponding to \(\beta\).

For each positive integer \(i\) we write \(\varepsilon_i\) for the element \((0,\dots,0,\underset i1,0,\dots)\) of \(\Gamma^+\).

\begin{theorem}[{\cite[4.5]{Holmes2018}}]
\label{thm:ReciprocalIrreducible}
  For every \(\alpha\vdash m\) and every \(\beta\models m-1\), we have
  \[
  \sum_{\substack{i=1\\\alpha_i>\alpha_{i+1}}}^{l(\alpha)}(\alpha_i-i)\chi_{\alpha-\varepsilon_i}(\beta)=\sum_{j=1}^{l(\beta)}\beta_j\chi_\alpha(\beta+\varepsilon_j).
  \]
  \hfill\qed
\end{theorem}

\begin{theorem}[{(Branching theorem)\ \cite[2.4.3]{MR644144}}]
\label{thm:BranchingTheorem}
  For every \(\alpha\vdash m\) and every \(\tau\in S_{m-1}\) we have
  \[
  \chi_\alpha(\tau)=\sum^{l(\alpha)}_{\substack{i=1\\\alpha_i>\alpha_{i+1}}}\chi_{\alpha-\varepsilon_i}(\tau).
  \]
  \hfill\qed
\end{theorem}

For \(\alpha\vdash m\), put
\[
T'_\alpha=\sum_{\sigma\in S_m}\chi_\alpha(\sigma)P_m(\sigma).
\]

\begin{theorem}
\label{thm:PartialTraceOfSymmetrizerPrime}
  For every \(\alpha\vdash m\), we have
  \[
  \tr_BT'_\alpha=\sum^{l(\alpha)}_{\substack{i=1\\\alpha_i>\alpha_{i+1}}}(n+\alpha_i-i)T'_{\alpha-\varepsilon_i}.
  \]
\end{theorem}
\begin{proof}
Let \(\alpha\vdash m\).  Using linearity of the partial trace, and then Lemmas \ref{lem:CosetReps} and \ref{lem:PartialTraceOfGroupElement}, we get
\begin{align*}
  \tr_B T'_\alpha&=\sum_{\sigma\in S_m}\chi_\alpha(\sigma)\tr_B(P_m(\sigma))\\
  &=\sum_{\tau\in S_{m-1}}\sum_{i=1}^m\chi_\alpha(\tau(i,m))\tr_B(P_m(\tau(i,m)))\\
  &=\sum_{\tau\in S_{m-1}}\sum_{i=1}^{m-1}\chi_\alpha(\tau(i,m))P_{m-1}(\tau)+n\sum_{\tau\in S_{m-1}}\chi_\alpha(\tau)P_{m-1}(\tau).
\end{align*}
Fix \(\tau\in S_{m-1}\) and let \(\tau=\tau_1\cdots\tau_t\) be a complete factorization of \(\tau\) as a product of disjoint cycles (including \(1\)-cycles, so that each \(1\le i\le m-1\) actually appears), and assume that the factors are in order of nonincreasing lengths.  Fix \(1\le i<m\).  We have \(i\in\tau_j\) (meaning \(i\) appears in \(\tau_j\)) for a uniquely determined \(j\).  If \(\tau_j\) has length \(l\), then \(\tau_j(i,m)\) is a cycle of length \(l+1\) since it is the same as \(\tau_j\) except with \(m\) inserted directly after \(i\).  In particular, if \(\beta\vdash(m-1)\) is the cycle partition of \(\tau\), it follows that \(\tau(i,m)\) is in the conjugacy class of \(S_m\) corresponding to \(\beta+\varepsilon_j\).  Therefore,
\begin{align*}
  \sum_{i=1}^{m-1}\chi_\alpha(\tau(i,m))&=\sum_{j=1}^t\sum_{i\in\tau_j}\chi_\alpha(\beta+\varepsilon_j)\\
  &=\sum_{j=1}^{l(\beta)}\beta_j\chi_\alpha(\beta+\varepsilon_j)\\
  &=\sum_{\substack{i=1\\\alpha_i>\alpha_{i+1}}}^{l(\alpha)}(\alpha_i-i)\chi_{\alpha-\varepsilon_i}(\tau),
\end{align*}
where the last step uses Theorem \ref{thm:ReciprocalIrreducible}.  Substituting this into the earlier equation and using Theorem \ref{thm:BranchingTheorem} in the second sum, we get
\begin{align*}
  \tr_B T'_\alpha&=\sum_{\tau\in S_{m-1}}\sum_{\substack{i=1\\\alpha_i>\alpha_{i+1}}}^{l(\alpha)}(\alpha_i-i)\chi_{\alpha-\varepsilon_i}(\tau)P_{m-1}(\tau)\\
  &\qquad+n\sum_{\tau\in S_{m-1}}\sum_{\substack{i=1\\\alpha_i>\alpha_{i+1}}}^{l(\alpha)}\chi_{\alpha-\varepsilon_i}(\tau)P_{m-1}(\tau)\\
  &=\sum^{l(\alpha)}_{\substack{i=1\\\alpha_i>\alpha_{i+1}}}(n+\alpha_i-i)\sum_{\tau\in S_{m-1}}\chi_{\alpha-\varepsilon_i}(\tau)P_{m-1}(\tau)\\
  &=\sum^{l(\alpha)}_{\substack{i=1\\\alpha_i>\alpha_{i+1}}}(n+\alpha_i-i)T'_{\alpha-\varepsilon_i}.
\end{align*}
\end{proof}

Since \(T_\alpha=(\chi_\alpha(1)/m!)T'_\alpha\) for \(\alpha\vdash m\), we get an immediate consequence.

\begin{corollary}
\label{cor:PartialTraceOfSymmetrizer}
  For every \(\alpha\vdash m\), we have
  \[
  \tr_BT_\alpha=\sum^{l(\alpha)}_{\substack{i=1\\\alpha_i>\alpha_{i+1}}}\frac{\chi_\alpha(1)(n+\alpha_i-i)}{m\chi_{\alpha-\varepsilon_i}(1)}T_{\alpha-\varepsilon_i}.
  \]
\end{corollary}

\section{Dimension of symmetry class}
\label{sec:DimensionOfSymmetryClass}

As an application of Theorem \ref{thm:PartialTraceOfSymmetrizerPrime} we give a proof by induction of the well-known formula for the dimension of the space \(V^{\chi_\alpha}:=T_\alpha(V^{\otimes m})\) (\(\alpha\vdash m\)), which is known as the \emph{symmetry class} of tensors corresponding to the irreducible character \(\chi_\alpha\) of \(S_m\) and the vector space \(V\).

\begin{theorem}[\cite{MR1475219}]
\label{thm:DimensionOfSymmetryClass}
  For every \(\alpha\vdash m\), we have
  \[
  \dim V^{\chi_\alpha}=\frac{(\chi_\alpha(1))^2}{m!}\prod_{i=1}^{l(\alpha)}\prod_{j=1}^{\alpha_i}(n+j-i),
  \]
  where \(n=\dim V\).
\end{theorem}
\begin{proof}
Let \(\alpha\vdash m\).  Since \(T_\alpha^2=T_\alpha\) \cite[Theorem 6.3]{MR1475219}, every eigenvalue of \(T_\alpha\) is either \(0\) or \(1\), so  \(\dim V^{\chi_\alpha}\) is the rank of \(T_\alpha\), which equals the trace of \(T_\alpha\):
\[
\dim V^{\chi_\alpha}=\tr T_\alpha=\frac{\chi_\alpha(1)}{m!}\tr T'_\alpha.
\]
Therefore, it suffices to show that
\[
\tr T'_\alpha=\chi_\alpha(1)\prod_{i=1}^{l(\alpha)}\prod_{j=1}^{\alpha_i}(n+j-i).
\]
We proceed by induction on \(m\).  If \(m=1\), then \(\alpha=(1,0,0,\dots)\) and \(\chi_\alpha\) is the trivial character, implying that \(T'_\alpha\) is the identity map on \(V\) and both sides of the equation equal \(n\).  Assume that \(m>1\).  Using Lemma \ref{lem:TraceIsTraceOfPartialTrace}, Theorem \ref{thm:PartialTraceOfSymmetrizerPrime}, and the induction hypothesis in turn, we get
\begin{align*}
  \tr T'_\alpha&=\tr(\tr_B T'_\alpha)\\
  &=\sum^{l(\alpha)}_{\substack{k=1\\\alpha_k>\alpha_{k+1}}}(n+\alpha_k-k)\tr(T'_{\alpha-\varepsilon_k})\\
  &=\sum^{l(\alpha)}_{\substack{k=1\\\alpha_k>\alpha_{k+1}}}(n+\alpha_k-k)\chi_{\alpha-\varepsilon_k}(1)\prod_{i=1}^{l(\alpha-\varepsilon_k)}\prod_{j=1}^{(\alpha-\varepsilon_k)_i}(n+j-i).
\end{align*}
Now, for each \(1\le k\le l(\alpha)\) with \(\alpha_k>\alpha_{k+1}\),
\begin{align*}
  (n+\alpha_k-k)\prod_{i=1}^{l(\alpha-\varepsilon_k)}&\prod_{j=1}^{(\alpha-\varepsilon_k)_i}(n+j-i)\\
  &=(n+\alpha_k-k)\prod_{\substack{i=1\\i\not=k}}^{l(\alpha)}\prod_{j=1}^{\alpha_i}(n+j-i)\prod_{j=1}^{\alpha_k-1}(n+j-k)\\
  &=\prod_{i=1}^{l(\alpha)}\prod_{j=1}^{\alpha_i}(n+j-i),
\end{align*}
so substituting into the earlier formula we get
\begin{align*}
  \tr T'_\alpha&=\sum^{l(\alpha)}_{\substack{k=1\\\alpha_k>\alpha_{k+1}}}\chi_{\alpha-\varepsilon_k}(1)\prod_{i=1}^{l(\alpha)}\prod_{j=1}^{\alpha_i}(n+j-i)\\
  &=\chi_\alpha(1)\prod_{i=1}^{l(\alpha)}\prod_{j=1}^{\alpha_i}(n+j-i),
\end{align*}
where the last step uses Theorem \ref{thm:BranchingTheorem}.  This completes the proof.
\end{proof}

\bibliographystyle{amsalpha}
\bibliography{HolmesBib}

\end{document}